\documentclass[11pt]{amsart}
\usepackage[T1]{fontenc}
\usepackage{amssymb, stmaryrd}
\usepackage{url}
\usepackage{bbm}
\usepackage[dvipsnames]{xcolor}
\usepackage{mathtools}

\usepackage{tikz}
\usetikzlibrary{arrows}
\usetikzlibrary{cd}
\tikzset{>=latex}

\usepackage{hyperref}
\usepackage{amsrefs}

\newcommand{\cat}{\mathcal{C}}

\DeclareMathOperator{\Hom}{Hom}
\DeclareMathOperator{\Aut}{Aut}
\newcommand{\catname}[1]{\mathbf{#1}}
\newcommand{\Span}{\catname{Span}}

\newcommand{\set}{\catname{Set}}
\newcommand{\cob}{\catname{Cob}}
\newcommand{\PMon}{\catname{PMon}}
\newcommand{\cut}{\mathsf{Cut}}
\newcommand{\id}{\mathrm{id}}

\newcommand{\op}{\mathrm{op}}

\renewcommand{\emptyset}{\varnothing}
\newcommand{\suchthat}{\;\middle|\;}
\newcommand{\Z}{\mathbb{Z}}

\newcommand{\kk}{\Bbbk}
\newcommand{\bitimes}[2]{\,\vphantom{\times}_{#1}\!\! \times_{#2}\!}
\DeclareMathOperator{\im}{im}

\newtheorem{thm}{Theorem}[section]
\newtheorem{prop}[thm]{Proposition}
\newtheorem{lemma}[thm]{Lemma}
\newtheorem{cor}[thm]{Corollary}

\theoremstyle{definition}
\newtheorem{definition}[thm]{Definition}

\newtheorem{remark}[thm]{Remark}

\newtheorem{example}[thm]{Example}

\newtheorem*{ack}{Acknowledgements}
\numberwithin{equation}{section}

\begin{document}

\title{Coherent span-valued 2D TQFTs}
\author{Sophia E Marx}
\address{Department of Mathematics\\University of Massachusetts\\Amherst, MA}
\email{semarx@umass.edu}
\author{Rajan Amit Mehta}
\address{Department of Mathematical Sciences\\Smith College\\44 College Lane\\Northampton, MA 01063}
\email{rmehta@smith.edu}


\begin{abstract}
We consider commutative Frobenius pseudomonoids in the bicategory of spans, and we show that they are in correspondence with $2$-Segal cosymmetric sets. Such a structure can be interpreted as a coherent $2$-dimensional topological quantum field theory taking values in the bicategory of spans. We also describe a construction that produces a $2$-Segal cosymmetric set from any partial monoid equipped with a distinguished element.
\end{abstract}

\maketitle

\section{Introduction}

A topological quantum field theory (TQFT) is defined to be a symmetric monoidal functor from a cobordism category to the category of vector spaces. More generally, one can consider $\cat$-valued TQFTs, where $\cat$ is any symmetric monoidal category. 

One of the first results that children learn about TQFTs is that commutative Frobenius algebras are in correspondence with $2$-dimensional oriented TQFTs \cites{abrams, dijkgraaf:thesis}. This result is really a theorem about the structure of the $2$-dimensional oriented cobordism category, so it immediately generalizes to a correspondence between commutative Frobenius objects in $\cat$ and $\cat$-valued $2$-dimensional oriented TQFTs.

We are particularly interested in TQFTs valued in categories of spans. Such structures naturally arise from sigma models in mathematical physics, due to the fact that fields can be restricted from a cobordism to its boundary components (see, for example, \cites{CHS,freed:remarks, haugseng_iterated_2018}). Another motivation is that span categories can be viewed as objectifications of the category of vector spaces (see, for example, \cites{BaezGroupoid,GKT_HLA}).

In order to properly define a category of spans, one needs to use isomorphism classes of spans as the morphisms. A desire to work at the level of spans themselves leads us to consider higher categories. When studying algebraic structures in span categories, one might then look for lifts to coherent structures in the higher categories. This is partly motivated by the aesthetic argument that any ``natural'' example should arise from such a coherent structure. We will work in the bicategory of spans of sets, which we denote as $\Span$, but we believe that analogous results hold in more general settings. 

In $\Span$, the coherent structure that lifts a Frobenius object is known as a \emph{Frobenius pseudomonoid} \cite{Street:Frobenius}. Following the above discussion, we view a commutative Frobenius pseudomonoid in $\Span$ as a coherent $\Span$-valued $2$-dimensional oriented TQFT. The purpose of this paper is to give an equivalent characterization of such objects in terms of simplicial sets with additional structure.

As a special case of a result of Stern \cite{Stern:span}, it is known that pseudomonoids in $\Span$ correspond to simplicial sets satisfying the \emph{$2$-Segal conditions} of \cites{Dyckerhoff-Kapranov:Higher, GKT1}; see \cites{MM:2segal,stern:perspectives} for recent expositions of this correspondence. In \cite{CMS}, commutative and Frobenius structures on pseudomonoids in $\Span$ were separately considered and found to correspond to additional structures on the corresponding simplicial sets. Specifically, it was shown that commutative pseudomonoids in $\Span$ correspond to $2$-Segal $\Gamma$-sets, and that Frobenius pseudomonoids in $\Span$ correspond to $2$-Segal paracyclic sets. 

The main result of this paper is to show that when $\Gamma$ and paracyclic structures are both present on the same underlying $2$-Segal set, then they fit together (with no further compatibility conditions) to form a unique \emph{cosymmetric} structure. In light of the results of \cite{CMS}, we conclude that commutative Frobenius pseudomonoids in $\Span$ correspond to $2$-Segal cosymmetric sets.

Additionally, we describe a construction that can produce $2$-Segal cosymmetric sets from fairly common data, such as a commutative partial monoid with an arbitrary distinguished element. Using this construction, we obtain numerous examples, including some that to our knowledge have not previously appeared in the $2$-Segal literature. As an application, we construct a cosymmetric structure on the nerve of an effect algebra, slightly refining a result of Roumen \cite{roumen:effectalgebras}, who showed that the nerve of an effect algebra has a cyclic structure.

We end the introduction with an additional TQFT-related motivation that suggests further directions. Implicit in the work of Crane and Yetter \cite{CraneYetter} is the fact that in the $321$-dimensional extended cobordism bicategory, the circle $S^1$ has the structure of a commutative Frobenius pseudomonoid; this structure can be seen more explicitly in the generator-and-relation description of \cite{BDSV}. Thus, a commutative Frobenius pseudomonoid in $\Span$ would form part of the data of a $\Span$-valued $321$-dimensional extended TQFT. Using the additional generators and relations in \cite{BDSV}, one might hope to build on our work and obtain a concise description of modular tensor objects in $\Span$.

\begin{ack}
We would like to thank Ivan Contreras and Walker Stern for many helpful discussions on topics related to this paper.
\end{ack}

\section{Structured simplicial sets}
\label{sec:structured}
The main results of this paper deal with simplicial sets that have certain additional structures and the relationships between those structures. These structures and the relationships between them can be summarized by the following commutative diagram\footnote{This diagram is known to people working in cyclic theory; for example, it appears (without $\Lambda_\infty$) in \cite{Loday}*{Section 6.4.2}.} of categories:

\begin{equation}\label{eqn:categories}
\begin{tikzcd}
    \Delta \arrow[r] \arrow[d] & \Lambda_\infty \arrow[r] & \Lambda \arrow[d] \\
    \Phi_*^\op \arrow[rr] && \Phi^\op
\end{tikzcd}
\end{equation}
Here, $\Delta$ is the simplex category, $\Lambda$ and $\Lambda_\infty$ are the cyclic and paracyclic categories, and $\Phi$ and $\Phi_*$ are skeletons of the categories of nonempty finite sets and finite pointed sets, respectively. In this section we will describe the categories and functors in \eqref{eqn:categories}, as well as the induced relationships between presheaves on these categories, in more detail.

\subsection{The simplex category}
To establish notation and terminology, we briefly review the simplex category and simplicial sets. This material is standard and can be found in, e.g., \cite{goerss-jardine}.

The \emph{simplex category} $\Delta$ is defined as follows. The objects are the sets $[n] = \{0,1,\dots, n\}$ for $n = 0,1, \dots$, and the morphisms are weakly order-preserving maps.

The morphisms in $\Delta$ are generated by \emph{coface maps} $\delta_i^n: [n-1] \to [n]$ for $0 \leq i \leq n$, given by
\[ \delta_i^n(k) = \begin{cases}
    k, & k < i,\\
    k+1 & k \geq i,
\end{cases}
\]
and \emph{codegeneracy maps} $\sigma_i^n: [n+1] \to [n]$ for $0 \leq i \leq n$, given by
\[ \sigma_i^n(k) = \begin{cases}
    k, & k \leq i, \\
    k-1, & k > i,
\end{cases}
\]
satisfying the \emph{cosimplicial relations}
\begin{align*}
    \delta_j^{n} \delta_i^{n-1} &= \delta_i^{n} \delta_{j-1}^{n-1},  \hspace{28px} i < j, \\
    \sigma_j^{n} \delta_i^{n-1} &= \begin{cases}
        \delta_i^{n+1} \sigma_{j-1}^n, & i < j,\\
        \id, &  i=j, j+1,\\
        \delta_{i-1}^{n+1} \sigma_j^n, &  i > j+1,
    \end{cases}\\
    \sigma_j^{n} \sigma_i^{n+1} &= \sigma_i^{n} \sigma_{j+1}^{n+1}, \hspace{26px} i \leq j.
\end{align*}

A \emph{simplicial set} is a presheaf on $\Delta$, i.e.\ a functor $X_\bullet: \Delta^\op \to \set$. From the generator-and-relation description of $\Delta$, it follows that a simplicial set
can be described by the following data:
\begin{itemize}
    \item sets $X_n$ for $n=0,1,\dots$,
    \item \emph{face maps} $d_i^n: X_n \to X_{n-1}$ for $0 \leq i \leq n$, and
    \item \emph{degeneracy maps} $s_i^n: X_n \to X_{n+1}$ for $0 \leq i \leq n$,
\end{itemize}
satisfying the \emph{simplicial relations}
\begin{align*}
    d_i^{n-1}d_j^n &= d_{j-1}^{n-1} d_i^n, \hspace{28px} i < j, \\
    d_i^{n-1} s_j^n &= \begin{cases}
    s_{j-1}^n d_i^{n+1}, & i < j,\\
    \id, & i=j,j+1,\\
    s_j^n d_{i-1}^{n+1}, & i > j+1,
    \end{cases}\\
    s_i^{n+1} s_j^n &= s_{j+1}^{n+1} s_i^n, \hspace{28px} i \leq j.
\end{align*}

Let $\cat$ be a category with the same objects as the simplex category, equipped with a functor $\Delta \to \cat$. Then we can view a presheaf on $\cat$, i.e\ a functor $\cat^\op \to \set$, as a \emph{structured simplicial set}, since (at least in principle) such a presheaf can be described as a simplicial set with additional structure maps and relations. All of the categories in \eqref{eqn:categories} are of this type, and in the remainder of this section we will describe each of them, giving generator-and-relation descriptions that explicitly allow us to realize their presheaves as simplicial sets carrying various types of additional structure.

\subsection{The paracyclic category}
In this section, we review the definition and basic properties of the paracyclic category and paracyclic sets. For more details as well as the original motivation for paracyclic objects, which comes from cyclic homology, see \cites{fiedorowicz-loday, getzler-jones}.

The \emph{paracyclic category} $\Lambda_\infty$ has the same objects as the simplex category, and a morphism $f \in \Hom_{\Lambda_\infty}([m],[n])$ is a weakly order-preserving map $f: \Z \to \Z$ such that 
\begin{equation}\label{eqn:equivariance}
    f(x+m+1) = f(x)+n+1.
\end{equation}
From \eqref{eqn:equivariance}, it follows that any map $f: \{0,\dots,m\} \to \Z$ that is weakly order-preserving and such that $f(m) \leq f(0)+n+1$ uniquely extends to a morphism $f \in \Hom_{\Lambda_\infty}([m],[n])$. In particular, there is a natural inclusion $\Delta \hookrightarrow \Lambda_\infty$, with the image consisting of morphisms $f$ such that $f(\{0,\dots,m\}) \subseteq \{0,\dots,n\}$.

For each $n$, we define the \emph{translation map} $T^n \in \Aut_{\Lambda_\infty}([n])$ by $T^n(i) = i+1$ for $i \in \Z$. Every $f \in \Hom_{\Lambda_\infty}([m],[n])$ can be uniquely factored as $f = g \circ (T^m)^k$, where $g \in \Hom_\Delta([m],[n])$ and $k \in \Z$ (see \cite{CMS}*{Proposition 4.1}), so it follows that the maps $T^n$, together with the coface and codegeneracy maps, generate $\Lambda_\infty$.  It also follows that
\[ \Aut_{\Lambda_\infty}([n]) \cong \Z,\]
with $T^n$ as a generator.

The translation maps satisfy the following relations which, together with the cosimplicial relations, give us a generator-and-relation description of $\Lambda_\infty$.
\begin{align} \label{eqn:paracycliccoface}
    T^n \delta_i^n &= \begin{cases}
        \delta_{i+1}^n T^{n-1}, & 0 \leq i < n, \\
        \delta_0^n, & i=n,
    \end{cases}\\ \label{eqn:paracycliccodegen}
    T^n \sigma_i^n &= \begin{cases}
        \sigma_{i+1}^n T^{n+1}, & 0 \leq i < n, \\
        \sigma_0^n (T^{n+1})^2, & i=n.
    \end{cases}
\end{align}

Another special family of morphisms in $\Lambda_\infty$ are the \emph{extra codegeneracy maps} $\sigma_{n+1}^n = \sigma_0^n T^{n+1} \in \Hom_{\Lambda_\infty}([n+1],[n])$. They satisfy relations
\begin{align} 
\label{eqn:paracycliclift1}
    \sigma_{n+1}^n \delta_i^{n+1} &= \begin{cases}
        T^n, & i=0,\\
        \delta_i^n \sigma_n^{n-1}, & 0<i<n+1,\\
        \id, & i=n+1,
    \end{cases}\\ \label{eqn:paracycliclift2}
    \sigma_{n+1}^n \sigma_i^{n+1} &=
        \sigma_i^n \sigma_{n+2}^{n+1}, \;\;\;\;\;\; 0 \leq i \leq n+1.
\end{align}

A \emph{paracyclic set} is a presheaf on $\Lambda_{\infty}$, i.e.\ a functor $\Lambda_{\infty}^\op \to \set$. From the above generator-and-relation descriptions of $\Lambda_\infty$, we see that a paracyclic set is equivalent to a simplicial set $X_\bullet$ equipped with invertible maps $\tau^n: X_n \to X_n$ satisfying relations dual to \eqref{eqn:paracycliccoface}--\eqref{eqn:paracycliccodegen}:
\begin{align}\label{eqn:paracyclicface}
    d_i^n \tau^n &= \begin{cases}
        \tau^{n-1} d_{i+1}^n, & i<n,\\
        d_0^n, & i=n,
    \end{cases}\\ \label{eqn:paracyclicdegen}
    s_i^n \tau^n &= \begin{cases}
        \tau^{n+1} s_{i+1}^n, & i<n,\\
        (\tau^{n+1})^2 s_0^n, & i=n.
    \end{cases}
\end{align}
It is often useful to think of $\tau^n$ as generating a $\Z$-action on $X_n$ for each $n$.

A paracyclic set $X_\bullet$ possesses \emph{extra degeneracy maps} $s_{n+1}^n = \tau^{n+1} s_0^n: X_n \to X_{n+1}$, which satisfy relations dual to \eqref{eqn:paracycliclift1}--\eqref{eqn:paracycliclift2}:
\begin{align}\label{eqn:paracyclicgamma1}
    d_i^{n+1} s_{n+1}^n &= \begin{cases}
        \tau^n, & i=0,\\
        s_n^{n-1} d_i^n, & 0 < i < n+1, \\
        \id, & i = n+1,
    \end{cases}\\ \label{eqn:paracyclicgamma2}
    s_i^{n+1} s_{n+1}^n &= s_{n+2}^{n+1} s_i^n, \;\;\;\;\;\; 0 \leq i \leq n+1.
\end{align}

\subsection{The cyclic category}

The \emph{cyclic category} $\Lambda$, introduced by Connes \cite{connes:cyclic}, is the quotient of $\Lambda_\infty$ by the additional relations $(T^n)^{n+1} = \id$ for all $n$.

A \emph{cyclic set} is a presheaf on $\Lambda$, i.e.\ a functor $\Lambda^\op \to \set$. Thus, a cyclic set is equivalent to a paracyclic set such that $(\tau^n)^{n+1} = \id$ for all $n$. This relation implies that the $\Z$-action on $X_n$ descends to an action of $\Z/(n+1)\Z$ for each $n$.

\subsection{The category of finite pointed cardinals}
\label{sec:phistar}

We use $\Phi_*$ to denote the category of finite pointed cardinals. It has the same objects as the simplex category, and the morphisms $f \in \Hom_{\Phi_*}([m],[n])$ are maps $f:[m] \to [n]$ such that $f(0) = 0$. Because $\Phi_*$ is a skeleton of the category of finite pointed sets, it is often referred to in the literature by names such as $\catname{Fin}_*$ or $\catname{FinSet}_*$.

The functor $\Delta \to \Phi_*^\op$ in \eqref{eqn:categories} is more easily defined via its opposite functor, known as $\cut: \Delta^\op \to \Phi_*$. It is given by face maps $d_i^n: [n] \to [n-1]$, 
\[ d_i^n(k) = \begin{cases}
    k, & k \leq i, \\
    k-1, & k > i,
\end{cases}\]
for $0 \leq i < n$ and
\[ d_n^n(k) = \begin{cases}
    k, & k < n, \\
    0, & k =n,
\end{cases}\]
and degeneracy maps $s_i^n: [n] \to [n+1]$,
\[ s_i^n(k) = \begin{cases}
    k, & k \leq i, \\
    k+1, & k > i.
\end{cases}\]
Informally, $s_i^n$ is the map that skips $i+1$, $d_i^n$ is the map that collapses $i$ and $i+1$ together for $i < n$, and $d_n^n$ collapses $0$ and $n$ together.

In addition to the morphisms in the image of $\cut$, $\Phi_*$ contains transposition maps $\theta_i^n: [n] \to [n]$ for $1 \leq i \leq  n-1$ that swap $i$ and $i+1$, keeping all other elements fixed. They satisfy the \emph{Moore relations}
\begin{align}
    (\theta_i)^2 &= \id, \label{eqn:moore1} \\
    \theta_i \theta_j \theta_i &= \theta_j \theta_i \theta_j, \;\; i = j-1, \label{eqn:moore2} \\
    \theta_i \theta_j &= \theta_j \theta_i, \;\; i < j-1, \label{eqn:moore3}
\end{align}
as well as the \emph{mixed relations}
\begin{align}\label{eqn:mixed1b}
    \theta_i s_j &= \begin{cases}
        s_j \theta_i, & i < j, \\
        s_{i-1}, & i = j, \\
        s_{i}, & i = j+1 \\        
        s_j \theta_{i-1}, & i>j+1,
    \end{cases} \\
    \theta_i d_j &= \begin{cases}
        d_j \theta_i, & i < j-1,\\
        d_i \theta_{i+1} \theta_i, & i = j-1,\\
        d_{i+1}\theta_i \theta_{i+1}, & i=j,\\
        d_j \theta_{i+1}, & i>j,
    \end{cases} \label{eqn:thetad} \\
    d_i \theta_i &= d_i,\label{eqn:mixed3b} \\    \label{eqn:dn}
    d_n^n &= d_0^n \theta_1^n \cdots \theta_{n-1}^n.
\end{align}
The maps $s_i^n$, $d_i^n$, and $\theta_i^n$ generate $\Phi_*$, and the simplicial relations, the Moore relations, and the mixed relations are sufficient to generate all relations (see \cite{CMS}*{Section 5.4}).

A \emph{$\Gamma$-set} is a functor $\Phi_* \to \set$. This terminology arises from the fact that $\Phi_*^\op$ is isomorphic to the category $\Gamma$ introduced by Segal \cite{SegalCoh}, so a functor $\Phi_* \to \set$ can be viewed as a presheaf on $\Gamma$.

From the generator-and-relation description of $\Phi_*$, we have that a $\Gamma$-set is a simplicial set $X_\bullet$ equipped with maps $\theta_i^n: X_n \to X_n$, $1 \leq i \leq n-1$, satisfying \eqref{eqn:moore1}--\eqref{eqn:dn}. Equivalently, a $\Gamma$-set is a simplicial set $X_\bullet$ equipped with an action of the symmetric group $S_n \cong \Aut_{\Phi_*}([n])$ on $X_n$ for each $n$, such that the transpositions $\theta_i^n$ satisfy \eqref{eqn:mixed1b}--\eqref{eqn:dn}.

\subsection{The category of finite cardinals}
\label{sec:phi}

We use $\Phi$ to denote the category of nonempty finite cardinals, which is a skeleton of the category of nonempty finite sets. It is the full subcategory of $\set$ consisting of objects $[n] = \{0,\dots, n\}$ for $n \geq 0$. 

There is a natural inclusion $\Phi_* \hookrightarrow \Phi$, which can be interpreted as the skeletal version of the basepoint-forgetting functor from the category of finite pointed sets to the category of nonempty finite sets.

The image of $\Phi_*$ in $\Phi$ consists of morphisms that fix $0$. The cyclic permutations $\tau^n: [n] \to [n]$, given by
\[ \tau^n(k) = \begin{cases}
    n & k=0,\\
    k-1 & k \geq 1,
\end{cases}\]
form an important family of maps that are not in the image of $\Phi_*$.
They satisfy the relations
\begin{align}
    (\tau^n)^{n+1} &= \id, \label{eqn:cyclicperm}\\
    \theta_i^n \tau^n &= \begin{cases}
        \tau^n \theta_{i+1}^n, & 1 \leq i < n-1, \\
        (\tau^n)^2 \theta_1^n \cdots \theta_{n-1}^n, & i = n-1,
    \end{cases} \label{eqn:tautheta}
\end{align}
as well as the relations \eqref{eqn:paracyclicface}--\eqref{eqn:paracyclicdegen}. We note that \eqref{eqn:moore1}--\eqref{eqn:moore3} and \eqref{eqn:cyclicperm}--\eqref{eqn:tautheta} are exactly the relations for the symmetric group $S_{n+1} \cong \Aut_{\Phi}([n])$.

Any map $f \in \Hom_\Phi([m],[n])$ can be uniquely factored as $f = (\tau^n)^k \circ g$, where $g \in \Hom_{\Phi_*}([m],[n])$ and $k \in \Z_{n+1}$, so it follows that the above relations, together with the relations in $\Phi_*$, are sufficient to characterize $\Phi$.

Because the relations \eqref{eqn:paracyclicface}--\eqref{eqn:paracyclicdegen} and \eqref{eqn:cyclicperm} that characterize a cyclic structure appear in $\Phi$, we obtain a functor $\Lambda^\op \to \Phi$ whose opposite is, by definition, the functor $\Lambda \to \Phi^\op$ in \eqref{eqn:categories}.

We use the term \emph{cosymmetric set} to describe a functor $\Phi \to \set$. This terminology arises from the fact that functors $\Phi^\op \to \set$ are commonly called \emph{symmetric sets}.

From the generator-and-relation description of $\Phi$, we have that a cosymmetric set is a simplicial set $X_\bullet$ equipped with maps $\theta_i^n: X_n \to X_n$, $1 \leq i \leq n-1$, and $\tau^n: X_n \to X_n$, satisfying \eqref{eqn:paracyclicface}--\eqref{eqn:paracyclicdegen}, \eqref{eqn:moore1}--\eqref{eqn:dn}, and \eqref{eqn:cyclicperm}--\eqref{eqn:tautheta}. Equivalently, a cosymmetric set is a simplicial set $X_\bullet$ equipped with an action of the symmetric group $S_{n+1}$ on $X_n$ for each $n$, such that the transpositions $\theta_i^n$ and cyclic permutations $\tau^n$ satisfy \eqref{eqn:paracyclicface}--\eqref{eqn:paracyclicdegen} and \eqref{eqn:mixed1b}--\eqref{eqn:dn}.

\begin{remark}
    We warn that there are other commonly-studied structures that involve $S_{n+1}$ actions on simplicial sets, but where the compatibility conditions with the face and degeneracy maps are different than those of a cosymmetric set. For example, the nerve of a groupoid naturally admits $S_{n+1}$ actions giving it the structure of a symmetric set, i.e.\ a functor $\Phi^\op \to \set$ (see \cite{Grandis}). There is also the notion of a $\Delta\mathfrak{S}$-set, i.e.\ a presheaf on the symmetric crossed simplicial group $\Delta \mathfrak{S}$ (see, e.g., \cite{fiedorowicz-loday}*{Example 6}). As a specific point of comparison, cosymmetric sets satisfy the identity $d_0^2 \theta_1^2 = d_2^2$, whereas symmetric sets satisfy $d_0^2 \theta_1^2 = \tau^1 d_2^2$ and $\Delta\mathfrak{S}$-sets satisfy $d_0^2 \theta_1^2 = \tau^1 d_0^2$. 
\end{remark}

\section{\texorpdfstring{$2$-Segal sets}{2-Segal sets}}

The notion of \emph{$2$-Segal space} was introduced independently by Dyckerhoff and Kapranov \cite{Dyckerhoff-Kapranov:Higher} and by Galv\`ez-Carillo, Kock, and Tonks \cite{GKT1}, the latter using the name \emph{decomposition space}. In this section, we briefly review the definition (in the discrete setting, i.e.\ that of a \emph{$2$-Segal set}) as well as the motivation for considering $2$-Segal sets in the context of this paper.  

In addition to the original works cited above, there are now several good introductory papers \cites{bergner2024combinatorial, BOORS, MM:2segal, stern:perspectives} that focus on the case of $2$-Segal sets.

\subsection{Definition of \texorpdfstring{$2$}{2}-Segal sets}

Let $X_\bullet$ be a simplicial set. Then, for $0 < i < n$, the following diagrams commute as a consequence of the simplicial relations:
\begin{equation}\label{diag:2Segal}
\begin{tikzcd}
X_{n+1} \arrow[d, "d_{i+1}"'] \arrow[r, "d_0"] & X_n \arrow[d, "d_i"] \\
X_n  \arrow[r, "d_0"]                                 & X_{n-1}
\end{tikzcd}
\hspace{3em}
\begin{tikzcd}
X_{n+1} \arrow[d, "d_i"'] \arrow[r, "d_{n+1}"] & X_n \arrow[d, "d_i"] \\
X_n  \arrow[r, "d_n"]                                 & X_{n-1}
\end{tikzcd}
\end{equation}

\begin{definition}\label{def:2Segal}
The simplicial set $X_\bullet$ is \emph{$2$-Segal} if the diagrams \eqref{diag:2Segal} are pullbacks for all $0<i<n$.
\end{definition}

\begin{remark}\label{remark:unitality}
The conditions in Definition \ref{def:2Segal} are taken from \cite{GKT1}*{Section 3.4}. There are extra conditions there, corresponding to what is called the \emph{unitality} property in \cite{Dyckerhoff-Kapranov:Higher}, but it was shown in \cite{FGKW:unital} that unitality holds automatically as a consequence of the conditions in Definition \eqref{def:2Segal}. There are several other equivalent forms of the definition, e.g.\ \cite{Dyckerhoff-Kapranov:Higher}*{Proposition 2.3.2} and \cite{BOORS2}*{Proposition 1.17}.
\end{remark}

The nerve of a category is $2$-Segal. In Section \ref{sec:partialmonoid}, we'll describe in some detail an example that is not necessarily the nerve of a category. More examples are in Sections \ref{sec:effect} and \ref{sec:noteffect}. The reader will find many other examples in the references mentioned above.

We will later make use of the following property, which combines the conditions from the definition with results from \cite{GKT1}*{Lemma 3.10}.
\begin{lemma}\label{lemma:nnpullback}
If $X_\bullet$ is $2$-Segal, then for $0 \leq i < j < n$ or $0 < i < j \leq n$, the diagram
\[
\begin{tikzcd}
X_{n+1} \arrow[r, "d_i"] \arrow[d, "d_{j+1}"'] & X_n \arrow[d, "d_j"] \\
X_n \arrow[r, "d_i"]                                  & X_{n-1}
\end{tikzcd}
\]
is a pullback.
\end{lemma}

\subsection{Motivation: \texorpdfstring{$2$}{2}-Segal sets are pseudomonoids in \texorpdfstring{$\Span$}{Span}}

Our motivation for considering $2$-Segal sets is a result of Stern \cite{Stern:span}, which implies that $2$-Segal sets are in correspondence with pseudomonoids in $\Span$. We very briefly review this correspondence here, and we refer the reader to \cites{CMS,MM:2segal,stern:perspectives} for more details. For the symmetric monoidal bicategory $\Span$, we refer to \cite{stay}*{Section 5} and \cite{Johnson-Yau}*{Example 2.1.22} for thorough treatments and to \cite{CMS}*{Section 1} and \cite{MM:2segal}*{Section 4} for brief expository accounts.

Suppose $X_\bullet$ is a simplicial set. Then we can form the spans
\begin{align}
\label{eqn:multandunit}
X_1 \times X_1 \xleftarrow{(d_2,d_0)} &X_2 \xrightarrow{d_1} X_1, & \{\mathrm{pt}\} \xleftarrow{} &X_0 \xrightarrow{s_0} X_1,
\end{align}
which can be viewed as multiplication and unit morphisms in $\Span$. From the diagrams \eqref{diag:2Segal} with $n=2$, we have maps
\begin{equation} \label{eqn:associator}
X_2 \bitimes{d_2}{d_1} X_2 \xleftarrow{(d_1,d_3)} X_3 \xrightarrow{(d_0,d_2)} X_2 \bitimes{d_1}{d_0} X_2.
\end{equation}
If $X_\bullet$ is $2$-Segal, then these maps are isomorphisms, and together they provide an associator for the multiplication. Additionally, we have maps
\begin{align}\label{eqn:unitors}
    X_1 &\xrightarrow{(d_1,s_0)} X_0 \bitimes{s_0}{d_2}X_2, & X_1 &\xrightarrow{(s_1,d_0)} X_2 \bitimes{d_0}{s_0}X_0.
\end{align}
If $X_\bullet$ is $2$-Segal, then these maps are isomorphisms as a result of the unitality property (see Remark \ref{remark:unitality}), and they provide left and right unitors. The pentagon equation can be deduced from \eqref{diag:2Segal} with $n=3$, and the triangle equation also follows from unitality.

The process we have just described produces an up-to-isomorphism correspondence between $2$-Segal sets and pseudomonoids in $\Span$. This correspondence can be upgraded to an equivalence of categories; see \cite{stern:perspectives}.

\subsection{More motivation: structured \texorpdfstring{$2$}{2}-Segal sets}

Suppose that $\cat$ is a category with the same objects as the simplex category, equipped with a functor $\Delta \to \cat$ that is the identity map on objects. We define a \emph{$2$-Segal $\cat$-set} to be a functor $\cat^\op \to \set$ such that the induced simplicial set is $2$-Segal. In particular, using the categories in \eqref{eqn:categories}, we can define $2$-Segal paracyclic sets, $2$-Segal cyclic sets, $2$-Segal $\Gamma$-sets, and $2$-Segal cosymmetric sets.

Using the correspondence between $2$-Segal sets and pseudomonoids in $\Span$, one could, informally, view a $2$-Segal $\cat$-set as a pseudomonoid in $\Span$ with additional structures and/or properties. Some precise statements along these lines are as follows.
\begin{itemize}
    \item $2$-Segal paracyclic sets correspond to Frobenius pseudomonoids in $\Span$ \cite{CMS}*{Theorem 4.2},
    \item $2$-Segal cyclic sets correspond to Calabi-Yau objects in $\Span$ (a Calabi-Yau object is the coherent structure corresponding to a symmetric Frobenius algebra) \cite{Stern:span}*{Theorem 3.29},
    \item $2$-Segal $\Gamma$-sets correspond to commutative pseudomonoids in $\Span$ \cite{CMS}*{Theorem 5.6}.
\end{itemize}

Given a $2$-Segal cosymmetric set, it is immediate from \eqref{eqn:categories} that there are induced paracyclic and $\Gamma$-sets with the same underlying $2$-Segal simplicial set. Via the above correspondences, it follows that a $2$-Segal cosymmetric set induces a pseudomonoid in $\Span$ with both Frobenius and commutative structures.

Our main result, which we will prove in Section \ref{sec:main}, provides the following converse result: given paracyclic and $\Gamma$-sets with the same underlying $2$-Segal simplicial set, there is a unique lift to a $2$-Segal cosymmetric set. Thus, there is a correspondence between $2$-Segal cosymmetric sets and commutative Frobenius pseudomonoids in $\Span$.

\subsection{Partial monoids and their nerves}\label{sec:partialmonoid}

A \emph{partial monoid} \cite{Segal:Conf} is a set $M$ equipped with a partially-defined operation $M_2 \to M$, $(x,y) \mapsto x\cdot y$, for $M_2 \subseteq M \times M$, satisfying the following:
\begin{enumerate}
    \item (Associativity) $(x\cdot y)\cdot z = x \cdot (y \cdot z)$ for all $x,y,z \in M$,
    \item (Identity) There exists an element $e \in M$ such that $e \cdot x = x \cdot e = x$ for all $x \in M$.
\end{enumerate}
In the above, equal signs are interpreted to mean that either both sides are defined and equal or both sides are undefined.

A morphism $\varphi:M \to N$ of partial monoids is a map such that, if $x \cdot y$ is defined for $x,y \in M$, then $\varphi(x) \cdot \varphi(y)$ is defined and equal to $\varphi(x\cdot y)$. Note that this does allow for the possibility that $x \cdot y$ is undefined but $\varphi(x) \cdot \varphi(y)$ is defined. We denote the category of partial monoids by $\PMon$.

Given a partial monoid, one can form a simplicial set $N_\bullet M$, called the \emph{nerve} of $M$, where $N_0 M$ has a single point, and where $N_n M$ consists of fully composable $n$-tuples:
\begin{equation} \label{eqn:partialmonoidnerve}
N_n M = \left\{ (x_1,\dots,x_n) \in M^n \suchthat x_1 \cdots x_n \text{ is defined}\right\}.
\end{equation}
The face and degeneracy maps are given by formulae that are identical to those for the nerve of a monoid.

It was shown in \cite{BOORS} that the nerve of a partial monoid is $2$-Segal, but we remark that the verification is formally identical to that of the nerve of a category, and in fact both constructions can be seen as special cases of the nerve of a \emph{partial category} (see \cite{MM:2segal}).

The following was observed in \cite{CMS}, but because it will play a significant role for us later, we will discuss it in a bit more depth.
\begin{prop}\label{prop:commpartialmonoidgamma}
If $M$ is a commutative partial monoid, then the permutation action of $S_n$ on $N_n M$ gives $N_\bullet M$ the structure of a $\Gamma$-set.  
\end{prop}

\begin{proof}
The statement can be proved by a straightforward check that the transpositions satisfy \eqref{eqn:mixed1b}--\eqref{eqn:dn}. However, it is also possible to directly construct $N_\bullet M$ as a $\Gamma$-set using the following approach, which is a special case of the \emph{plasmic nerve} construction of \cite{beardsley-nakamura:plasmas}. 

For each $n \geq 0$, let $P_n = P(\{1, \dots, n\})$ be the power set of $\{1,\dots, n\}$, with the commutative partial monoid structure given by disjoint union (see Example \ref{ex:disjointunion}). Any $f \in \Hom_{\Phi_*}([m],[n])$ induces a map $f^*: P_n \to P_m$, given by $f^*(A) = f^{-1}(A)$, which is a morphism of partial monoids. This defines a functor $\Phi_*^\op \to \PMon$ which, at the level of objects, takes $[n]$ to $P_n$.

Given a commutative partial monoid $M$, we can then define 
\[X_n = \Hom_{\PMon}(P_n, M),\]
and it immediately follows from the above that $X_\bullet$ has the structure of a $\Gamma$-set. It then remains to check that this definition of $X_n = N_n M$ as defined in \eqref{eqn:partialmonoidnerve}. To see this, we observe that a morphism $\phi:P_n \to M$ is uniquely given by elements $x_i = \phi(\{i\}) \in M$, for $i = 1, \dots, n$, such that $\prod x_i$ is defined. One can then derive the formulae for face and degeneracy maps from the definition of $\cut$ in Section \ref{sec:phistar}. 
\end{proof}

\begin{remark}
The construction described in the proof of Proposition \ref{prop:commpartialmonoidgamma} produces a $2$-Segal $\Gamma$-set  even when $M$ is noncommutative. In this case, $X_\bullet$ is not isomorphic to the nerve of $M$. Rather, it is the simplicial subset of the nerve where $X_n$ consists of fully composable $n$-tuples $(x_1,\dots, x_n)$ such that $x_i x_j = x_j x_i$ for all $i,j$.
\end{remark}

We conclude with some illustrative examples of partial monoids.
\begin{example}
A monoid is a special case of a partial monoid.
\end{example}
\begin{example}\label{ex:L}
Let $L$ be a natural number, and let $M = \{0, 1, \dots, L\}$. Then the partial addition operation (defined when the sum is $\leq L$) gives $M$ the structure of a commutative partial monoid. Its nerve is given by
        \[ N_n M = \left\{(x_1,\dots, x_n) \in M^n \suchthat \sum x_i \leq L\right\}.\]
\end{example}
\begin{example}\label{ex:disjointunion}
Let $S$ be a set, and let $M = P(S)$ be the power set of $S$. For $A, B \in P(S)$, we say that the disjoint union $A \sqcup B$ is undefined if $A \cap B \neq \emptyset$, and $A \sqcup B = A \cup B$ if $A \cap B = \emptyset$. The disjoint union operation gives $M$ the structure of a commutative partial monoid. Its nerve is given by
        \[ N_n M = \left\{(A_1, \dots, A_n) \in M^n \suchthat A_i \cap A_j = \emptyset \text{ for all } i \neq j\right\}.\]
\end{example}
\begin{example}\label{ex:directsum}
Let $V$ be a vector space, and let $M$ be the set of subspaces of $V$. For $W_1, W_2 \in M$, we say that the direct sum $W_1 \oplus W_2$ is undefined if $W_1 \cap W_2 \neq \{0\}$, and $W_1 \oplus W_2 = W_1 + W_2$ if $W_1 \cap W_2 = \{0\}$. The direct sum operation gives $M$ the structure of a commutative partial monoid. Its nerve is given by
        \[ N_n M = \left\{(W_1,\dots,W_n) \in M^n \suchthat W_i \cap \left({\textstyle \bigoplus_{j \neq i}} W_j\right) = \{0\} \text{ for all } i\right\}.\]
\end{example}
\begin{example}\label{ex:orthogonalsum}
Let $\mathcal{H}$ be a Hilbert space, and let $M$ be the set of subspaces of $\mathcal{H}$. For $W_1, W_2 \in M$, we say that the orthogonal sum $W_1 \boxplus W_2$ is undefined if $W_1 \not\perp W_2$, and $W_1 \boxplus W_2 = W_1 \oplus W_2$ if $W_1 \perp W_2$. The orthogonal sum operation gives $M$ the structure of a commutative partial monoid. Its nerve is given by
        \[ N_n M = \left\{(W_1,\dots,W_n) \in M^n \suchthat W_i \perp W_j \text{ for all } i\neq j \right\}.\]
\end{example}

\section{Main result}
\label{sec:main}

\begin{thm}\label{thm:main}
Let $P: \Lambda_\infty^\op \to \set$ be a paracyclic set, and let $Q: \Phi_* \to \set$ be a $\Gamma$-set, and suppose that $P$ and $Q$ have the same underlying $2$-Segal simplicial set. Then $P$ is cyclic, and there is a unique cosymmetric set $F: \Phi \to \set$ such that the diagram
\begin{center}
\begin{tikzcd}
                                 & \Lambda^\op \arrow[rd] \arrow[rrd, "P"', bend left] &                     &      \\
\Delta^\op \arrow[ru] \arrow[rd] &                                         & \Phi \arrow[r, "F"] & \set \\
                                 & \Phi_* \arrow[ru] \arrow[rru, "Q", bend right]     &                     &     
\end{tikzcd}
\end{center}
commutes.
\end{thm}

Combining Theorem \ref{thm:main} with the results of \cite{CMS}, we immediately have the following.
\begin{cor}\label{cor:commfrob}
    There is a one-to-one correspondence (up to isomorphism) between $2$-Segal cosymmetric sets and commutative Frobenius pseudomonoids in the bicategory $\Span$.
\end{cor}

The remainder of this section will be devoted to proving Theorem \ref{thm:main}. Uniqueness is immediate from the presentations given in Section \ref{sec:structured}, since the images of $\Lambda^\op$ and $\Phi_*$ contain all of the generators of $\Phi$, so what needs to be shown is existence.

Suppose that $X_\bullet$ is a $2$-Segal set equipped with the following:
\begin{itemize}
    \item a paracyclic structure, i.e.\ invertible maps $\tau^n: X_n \to X_n$ satisfying \eqref{eqn:paracyclicface}--\eqref{eqn:paracyclicdegen};
    \item a $\Gamma$-structure, i.e.\ maps $\theta_i^n: X_n \to X_n$ for $1 \leq i \leq n-1$ satisfying \eqref{eqn:moore1}--\eqref{eqn:dn}.
\end{itemize}
To prove Theorem \ref{thm:main}, we need to show that the paracyclic structure is cyclic, i.e.\ that \eqref{eqn:cyclicperm} holds, and that the compatibility conditions \eqref{eqn:tautheta} hold.

We first observe that the extra degeneracy maps in a $2$-Segal paracyclic set satisfy analogues of the unitality property. The $n=1$, $i=1$ case of the following lemma appeared in \cite{CMS}*{Lemma 4.3}, although the proof was omitted there.
\begin{lemma}
\label{lemma:paracyclicpullback}
For all $1 \leq i \leq n$, the diagram
    \[
    \begin{tikzcd}
        X_n \arrow[r,"s_{n+1}"] \arrow[d,"d_i"] & X_{n+1} \arrow[d,"d_i"] \\
        X_{n-1} \arrow[r,"s_n"] & X_n
    \end{tikzcd}
    \]
    is a pullback.
\end{lemma}
\begin{proof}
    The proof is formally identical to the proof of \cite{FGKW:unital}*{Proposition 2.1}, but we write it out for the sake of completeness.

    For $1 \leq i < n$, consider the following diagram.
    \[
    \begin{tikzcd}
        X_n \arrow[r,"s_{n+1}"] \arrow[d,"d_i"] & X_{n+1} \arrow[d,"d_i"] \arrow[r,"d_{n+1}"] & X_n \arrow[d,"d_i"]\\
        X_{n-1} \arrow[r,"s_n"] & X_n \arrow[r,"d_n"] & X_{n-1}
    \end{tikzcd}
    \]
    By Lemma \ref{lemma:nnpullback}, the square on the right is a pullback. The compositions of the horizontal maps are identity maps, so the big square is a pullback. By the Prism Lemma, it follows that the left square is a pullback.

For the case $i=n$, consider the following diagram.
\[
\begin{tikzcd}[sep=small]
X_n \arrow[rd, "s_{n+1}"] \arrow[dd, "d_n"'] \arrow[rrr, "s_{n+1}"] & &  & X_{n+1} \arrow[rd, "s_{n+2}"] \arrow[dd, "d_n"' near start] \arrow[rrr, "d_{n+1}"] &  &  & X_n \arrow[rd, "s_{n+1}"] \arrow[dd, "d_n"' near start] & \\
& X_{n+1}  \arrow[rrr, crossing over, "s_{n+1}" near start] & & & X_{n+2}  \arrow[rrr, crossing over, "d_{n+1}" near start] &  &  & X_{n+1} \arrow[dd,"d_n"] \\
X_{n-1} \arrow[rd, "s_n"'] \arrow[rrr, "s_n"' near end]  &  &  & X_n \arrow[rd, "s_{n+1}"'] \arrow[rrr, "d_n"' near end]  &  & & X_{n-1} \arrow[rd, "s_n"']  &   \\ 
& X_n \arrow[from=uu, crossing over, "d_n" near start] \arrow[rrr, "s_n"']  &  &    & X_{n+1} \arrow[from=uu, crossing over, "d_n" near start] \arrow[rrr, "d_n"'] &  & & X_n
\end{tikzcd}
\]
This diagram presents the square on the left and right ends as a retract of the middle square, which we have already shown is a pullback. Since retracts of pullback squares are pullbacks, the result follows.
\end{proof}
\begin{remark}\label{remark:paraunital}
If $\omega \in X_{n+1}$ is such that $d_i \omega \in \im(s_n^{n-1})$ for some $1 \leq i \leq n$, then Lemma \ref{lemma:paracyclicpullback} implies that $\omega \in \im(s_{n+1}^n)$. Writing $\omega = s_{n+1} \eta$, we can solve for $\eta$ by applying $d_{n+1}$ to both sides, giving us
\[ \omega = s_{n+1} d_{n+1} \omega.\]
\end{remark}

Next, we will deduce some interactions that arise between the maps $s_{n+1}^n$ and $\tau^n$, coming from the paracyclic structure, and the maps $\theta_i^n$, coming from the $\Gamma$-structure.
\begin{lemma}\label{lemma:stautheta}
    The following are satisfied:
    \begin{enumerate}
        \item $\theta_i^{n+1} s_{n+1}^n = s_{n+1}^n d_{n+1}^{n+1} \theta_i^{n+1} s_{n+1}^n$ for $1 \leq i \leq n$,
        \item $\theta_i^{n+1} s_{n+1}^n = s_{n+1}^n \theta_i^n$ for $1 \leq i \leq n-1$,
        \item $\theta_n^{n+1} \cdots \theta_1^{n+1} s_{n+1}^n = s_{n+1}^n \tau^n$,
    \end{enumerate}
\end{lemma}
\begin{proof}
For $\omega \in X_n$, consider $\theta_i s_{n+1} \omega \in X_{n+1}$ for any $1 \leq i \leq n$. Since $d_i \theta_i s_{n+1} \omega= d_i s_{n+1} \omega = s_n d_i \omega$, we have by Remark \ref{remark:paraunital} that
\[ \theta_i s_{n+1} \omega = s_{n+1}d_{n+1} \theta_i s_{n+1} \omega,\]
which proves the first statement.

Using the first statement, together with \eqref{eqn:thetad} and \eqref{eqn:paracyclicgamma1}, we have, when $i \leq n-1$,
\begin{align*}
    \theta_i^{n+1} s_{n+1}^n &= s_{n+1}^n d_{n+1}^{n+1} \theta_i^{n+1} s_{n+1}^n \\
    &= s_{n+1}^n  \theta_i^{n+1} d_{n+1}^{n+1} s_{n+1}^n \\
    &= s_{n+1}^n  \theta_i^{n+1},  
\end{align*}
which proves the second statement.

Using the first and second statements, together with \eqref{eqn:dn}, we have
\begin{align*}
    \theta_n^{n+1} \cdots \theta_1^{n+1} s_{n+1}^n &= 
     \theta_n^{n+1} s_{n+1}^n  \theta_{n-1}^{n+1} \cdots \theta_1^{n+1} \\
&= s_{n+1}^n d_{n+1}^{n+1} \theta_n^{n+1} s_{n+1}^n \theta_{n-1}^{n+1} \cdots \theta_1^{n+1}\\
     &= s_{n+1}^n d_{n+1}^{n+1} \theta_n^{n+1} \cdots \theta_1^{n+1} s_{n+1}^n \\
    &= s_{n+1}^n d_0^{n+1} s_{n+1}^n \\
    &= s_{n+1}^n \tau^n,
\end{align*}
which proves the third statement.
\end{proof}

\begin{prop}\label{prop:cyclic}
    The paracyclic structure on $X_\bullet$ is cyclic, i.e.\ $(\tau^n)^{n+1} = \id$ for all $n \geq 1$. 
\end{prop}
\begin{proof}
We note that, in $\Phi_*$, the morphism $\theta_n^{n+1} \cdots \theta_1^{n+1}$ cyclically permutes the elements $\{1,\dots, n+1\} \subset [n+1]$, so the identity $(\theta_n^{n+1} \cdots \theta_1^{n+1})^{n+1} = \id$ holds. 

Using the third statement in Lemma \ref{lemma:stautheta}, we have
    \begin{align*}
s_{n+1}^n (\tau^n)^{n+1} &= (\theta_n^{n+1} \cdots \theta_1^{n+1})^{n+1} s_{n+1}^n \\
&= s_{n+1}^n.
    \end{align*}
The result follows from the fact that $s_{n+1}^n$ is injective.
\end{proof}

\begin{prop}\label{prop:cyclicswap}
    Equation \eqref{eqn:tautheta} is satisfied.
\end{prop}

\begin{proof}
    First we consider the case $1 \leq i < n-1$. Then, using Lemma \ref{lemma:stautheta} and the Moore relations \eqref{eqn:moore2}--\eqref{eqn:moore3}, we have
\begin{align*}
    s_{n+1}^n \tau^n \theta_{i+1}^n &= \theta_n^{n+1} \cdots \theta_1^{n+1} s_{n+1}^n \theta_{i+1}^n \\
    &= \theta_n^{n+1} \cdots \theta_1^{n+1} \theta_{i+1}^{n+1} s_{n+1}^n \\
    &= \theta_i^{n+1}\theta_n^{n+1} \cdots \theta_1^{n+1} s_{n+1}^n \\
    &= \theta_i^{n+1} s_{n+1}^n \tau^n \\
    &= s_{n+1}^n \theta_i^n \tau^n,
\end{align*}
so it follows that $\tau^n \theta_{i+1}^n = \theta_i^n \tau^n$.

For the $i=n-1$ case, we again use Lemma \ref{lemma:stautheta} and the Moore relations to get
\begin{align*}
    s_{n+1}^n (\tau^n)^2 \theta_1^n \cdots \theta_{n-1}^n 
    &= (\theta_n^{n+1} \cdots \theta_1^{n+1})^2 s_{n+1}^n \theta_1^n \cdots \theta_{n-1}^n \\
    &= \theta_n^{n+1} \cdots \theta_1^{n+1} \theta_n^{n+1} s_{n+1}^n \\
    &= \theta_{n-1}^{n+1} \theta_n^{n+1} \cdots \theta_1^{n+1} s_{n+1}^n \\
    &= \theta_{n-1}^{n+1} s_{n+1}^n \tau^n\\
    &= s_{n+1}^n \theta_{n-1}^n \tau^n. \qedhere
\end{align*}
\end{proof}
Together, Propositions \ref{prop:cyclic} and \ref{prop:cyclicswap} complete the proof of Theorem \ref{thm:main}.

\section{The simplex construction}\label{sec:simplex}
In this section, we describe a useful construction for producing $2$-Segal cosymmetric sets.

\subsection{The basepoint-adjoining functor}

Recall that $\Phi$ is a skeleton of the category of nonempty finite sets, and $\Phi_*$ is a skeleton of the category of finite pointed sets. In Section \ref{sec:phi}, we used the natural inclusion $\Phi_* \hookrightarrow \Phi$, corresponding to the basepoint-forgetting functor from finite pointed sets to nonempty finite sets. 

As is well-known, the basepoint-forgetting functor has a left adjoint, the basepoint-adjoining functor. On objects, the corresponding functor $\Phi \to \Phi_*$ takes $[n]$ to $[n+1]$, and on morphisms, it takes $f: [m] \to [n]$ to $\hat{f}: [m+1] \to [n+1]$, given by $\hat{f}(0) = 0$ and $\hat{f}(i) = f(i-1)+1$ for $1 \leq i \leq m+1$.

If we apply the basepoint-adjoining functor to the generators of $\Phi$, as described in Section \ref{sec:structured}, we can write the result in terms of the generators of $\Phi_*$ as follows:
\begin{align}
    d_i^n &\longmapsto \begin{cases}
        d_{i+1}^{n+1} & i<n,\\
        d_1^{n+1} \theta_2^{n+1} \cdots \theta_n^{n+1} & i=n,
    \end{cases} \label{eqn:adjoinface}\\
    s_i^n &\longmapsto s_{i+1}^{n+1}, \label{eqn:adjoindegen}\\
    \theta_i^n &\longmapsto \theta_{i+1}^{n+1}, \label{eqn:adjointheta}\\
    \tau^n &\longmapsto \theta_n^{n+1} \cdots \theta_1^{n+1}. \label{eqn:adjointau}
\end{align}

Given any $\Gamma$-set $X_\bullet$, we can compose with the basepoint-adjoining functor to obtain a cosymmetric set $\hat{X}_\bullet$, where $\hat{X}_n = X_{n+1}$. From \eqref{eqn:adjoinface}--\eqref{eqn:adjoindegen} we obtain formulae for the simplicial structure on $\hat{X}_\bullet$ in terms of the structure maps on $X_\bullet$:
\begin{align*}
\hat{d}_i^n &= \begin{cases}
    d_{i+1}^{n+1} & i < n,\\
    d_1^{n+1} \theta_2^{n+1} \cdots \theta_n^{n+1} & i=n,
\end{cases}\\
\hat{s}_i^n &= s_{i+1}^{n+1}.
\end{align*}
From \eqref{eqn:adjointheta}--\eqref{eqn:adjointau} we can similarly write explicit formulae for the additional cosymmetric structure on $X_\bullet$, but there is a more direct description. The $\Gamma$-structure on $X_\bullet$ provides $S_n$-actions on $X_n$ for each $n$, so $\hat{X}_n = X_{n+1}$ has an $S_{n+1}$-action, and this is exactly the $S_{n+1}$-action that gives the cosymmetric structure on $\hat{X}_\bullet$.

\begin{thm}\label{thm:basepointadjoin}
    Let $X_\bullet$ be a $\Gamma$-set. If $X_\bullet$ is $2$-Segal, then the cosymmetric set $\hat{X}_\bullet$ is $2$-Segal.
\end{thm}

\begin{proof}
The diagrams \eqref{diag:2Segal} for $\hat{X}_\bullet$, expressed in terms of the structure maps of $X_\bullet$, are as follows:
\[
\begin{tikzcd}[column sep=huge]
X_{n+2} \arrow[d, "d_{i+2}"'] \arrow[r, "d_1"] & X_{n+1} \arrow[d, "d_{i+1}"]\\
X_{n+1}   \arrow[r, "d_1"]                                  & X_n             
\end{tikzcd}
\hspace{3em}
\begin{tikzcd}[column sep=huge]
X_{n+2} \arrow[d, "d_{i+1}"'] \arrow[r, "d_1 \theta_2 \cdots \theta_{n+1}"] & X_{n+1} \arrow[d, "d_{i+1}"] \\
X_{n+1} \arrow[r, "d_1 \theta_2 \cdots \theta_n"]                              & X_n          
\end{tikzcd}
\]
The diagram on the left is a pullback by Lemma \ref{lemma:nnpullback}. The diagram on the right can be expanded as
\[\begin{tikzcd}[column sep=huge]
X_{n+2} \arrow[d, "d_{i+1}"'] \arrow[r, "\theta_2 \cdots \theta_{n+1}"] & X_{n+2} \arrow[d, "d_{i+2}"'] \arrow[r, "d_1"] & X_{n+1} \arrow[d, "d_{i+1}"] \\
X_{n+1} \arrow[r, "\theta_2 \cdots \theta_n"] & X_{n+1} \arrow[r, "d_1"]                              & X_n          
\end{tikzcd}
\]
which shows that it is a pullback because of Lemma \ref{lemma:nnpullback} and the fact that the $\theta_i$ maps are invertible.
\end{proof}

\begin{remark}
Readers familiar with the theory of $2$-Segal objects may recognize a similarity between the construction of $\hat{X}_n$ and the d\'ecalage functor, but we stress that the difference in the definition of $\hat{d}_n^n$ has nontrivial implications. In particular, whereas the d\'ecalage of a $2$-Segal object is $1$-Segal, we will see in examples that this is not the case for $\hat{X}_n$.
\end{remark}

\subsection{\texorpdfstring{Pointed $\Gamma$-sets and cosymmetric sets}{Pointed Gamma-sets and cosymmetric sets}}

A structured simplicial set $X_\bullet$ is said to be \emph{pointed} if $X_0$ has a single point. In both $\Phi_*$ and $\Phi$, the object $[0]$ is terminal, so every $\Gamma$-set (cosymmetric set) naturally decomposes as a disjoint union of pointed $\Gamma$-sets (cosymmetric sets), indexed by the elements of $X_0$. In this case, for each $u \in X_0$, we denote by $X_\bullet^u$ the restriction of $X_\bullet$ to the component lying over $u$.

The following result is a straightforward exercise.
\begin{prop}\label{prop:restriction}
A $\Gamma$-set (cosymmetric set) $X_\bullet$ is $2$-Segal if and only if $X_\bullet^u$ is $2$-Segal for all $u \in X_0$.
\end{prop}

\subsection{Simplex sets of commutative partial monoids} \label{sec:simplexdef}

Let $M$ be a commutative partial monoid. By Proposition \ref{prop:commpartialmonoidgamma}, the nerve of $M$ is a $2$-Segal $\Gamma$-set, where $S_n$ acts on $N_n M$ by permuting the components. Applying the basepoint-adjoining functor, we obtain the cosymmetric set $\hat{N}_\bullet M$, where
\[ \hat{N}_n M = N_{n+1} M = \left\{(x_0,\dots, x_n) \in M^{n+1} \suchthat x_0 \cdots x_n \text{ is defined}\right\},\]
the face and degeneracy maps are given by
\begin{align*}
    d_i(x_0, \dots, x_n) &= \begin{cases} (x_0, \dots, x_{i-1}, x_i \cdot x_{i+1}, x_{i+2}, \dots, x_n), & 0 \leq i \leq n-1,\\
    (x_0 \cdot x_n, x_1, \dots, x_{n-1}), & i=n,
    \end{cases}
    \\
    s_i(x_0, \dots, x_n) &= (x_0, \dots, x_i, e, x_{i+1}, \dots, x_n), \hspace{5em} 0 \leq i \leq n,
\end{align*}
and where $S_{n+1}$ acts by permuting the components. Theorem \ref{thm:basepointadjoin} tells us that $\hat{N}_\bullet M$ is $2$-Segal.

Alternatively, the cosymmetric structure on $\hat{N}_\bullet M$ can be directly described as follows. Given $f \in \Hom_{\Phi}([m],[n])$, let
\[f_*: \hat{N}_m M \to \hat{N}_n M\]
be given by $f_*(x_0,\dots,x_m) = (y_0,\dots,y_n)$, where
\[ y_i = \prod_{j \in f^{-1}(i)} x_j.\]
It is straightforward to check that this gives a functor $\Phi \to \set$, and that it agrees with the face and degeneracy maps given above.

Now, suppose we choose a distinguished element $L \in M$. Since $\hat{N}_0 M  = M$, we can consider the component $\hat{N}_\bullet^L M$, given by
\begin{equation} \label{eqn:test}
\hat{N}_n^L M = \left\{(x_0,\dots, x_n) \in M^{n+1} \suchthat x_0 \cdots x_n=L\right\},
\end{equation}
which inherits the cosymmetric structure and is $2$-Segal by Proposition \ref{prop:restriction}. We refer to $\hat{N}_\bullet^L M$ as the \emph{$L$-simplex set} of $M$.

We emphasize that, in general, the $L$-simplex set of $M$ is different from the nerve of $M$, although in some cases (for example, when $M$ is an abelian group) the two are isomorphic. The relationship between the $L$-simplex set and the nerve is further discussed in Section \ref{subsec:simplexvsnerve}.

The prototypical example is the following. Further examples are given in Sections \ref{sec:effect} and \ref{sec:noteffect}.

\begin{example}\label{ex:interval}
   Let $M$ be the interval $[0,1]$ with the partially-defined addition operation, and let $L = 1$. Then
   \[ \hat{N}^1_n M = \left\{(x_0, \dots, x_n) \in [0,1]^{n+1} \suchthat x_0 + \dots + x_n = 1 \right\}\]
   is the topological $n$-simplex, and the cosymmetric structure provides the usual $S_{n+1}$ action.
\end{example}

\subsection{Simplex sets vs.\ nerves}\label{subsec:simplexvsnerve}

Let $M$ be a commutative partial monoid equipped with a distinguished element $L \in M$. From this data two different associated simplicial sets --- the nerve $N_\bullet M$ and the $L$-simplex set $\hat{N}^L_\bullet M$ --- can be constructed. The nerve is a $\Gamma$-set, and the $L$-simplex set also has a $\Gamma$-set structure, induced from the cosymmetric structure.

The following can be seen directly from the description of the $L$-simplex set given in Section \ref{sec:simplexdef}.

\begin{prop}\label{prop:simplextonerve}
    The maps
    \begin{align*}
        \hat{N}_n^L M &\longrightarrow N_n M, \\
        (x_0, \dots, x_n) &\mapsto (x_1, \dots, x_n),
    \end{align*} 
    form a morphism of $\Gamma$-sets $\hat{N}_\bullet^L M \to N_\bullet M$.
\end{prop}

In general, the morphism in Proposition \ref{prop:simplextonerve} is not an isomorphism, but the following provides a necessary and sufficient condition.

\begin{definition}
A commutative partial monoid $M$ is said to have the \emph{orthocomplement property} with respect to $L \in M$ if, for every $x \in M$, there is a unique $x^\perp \in M$ such that $x \cdot x^\perp = L$.
\end{definition}

\begin{prop}\label{prop:nervetosimplex}
    The morphism in Proposition \ref{prop:simplextonerve} is an isomorphism if and only if $M$ has the orthocomplement property with respect to $L$.
\end{prop}
\begin{proof}
($\impliedby$) If $M$ has the orthocomplement property with respect to $L$, then, for $(x_0,\dots,x_n) \in \hat{N}_n^L M$, we necessarily have $x_0 = (x_1 \cdots x_n)^\perp$, so the maps
\begin{align*}
    N_n M &\longrightarrow \hat{N}_n^L M, \\
    (x_1, \dots, x_n) &\mapsto \left( (x_1 \cdots x_n)^\perp, x_1, \dots, x_n \right),
\end{align*}
provide an inverse to the map in Proposition \ref{prop:simplextonerve}.

($\implies$) For $n=1$, the map in Proposition \ref{prop:simplextonerve} is given by $(x_0, x_1) \mapsto x_1$. If this map is an isomorphism, then, for every $x_1 \in M$, there is a unique $x_0 \in M$ such that $(x_0, x_1) \in \hat{N}_1^L M$ or, equivalently, such that $x_0 \cdot x_1 = L$.
\end{proof}

\begin{cor}\label{cor:cyclicnerve}
    Let $M$ be a commutative partial monoid that has orthocomplements with respect to $L \in M$. Then the nerve of $M$ has a cosymmetric structure, and thus an induced cyclic structure.
\end{cor}
We stress that, in Corollary \ref{cor:cyclicnerve}, the cosymmetric and cyclic structures on the nerve depend on the choice of $L$. 
Explicitly, the cyclic structure on $N_n M$ is given by
\begin{equation}\label{eqn:tau}
\tau(x_1,\dots, x_n) = (x_2, \dots, x_n, (x_1 \cdots x_n)^\perp).
\end{equation} 

\subsection{Example: effect algebras}
\label{sec:effect}

The notion of \emph{effect algebra} was introduced by Foulis and Bennett \cite{foulis-bennett} in the study of quantum logic.

\begin{definition}
    An \emph{effect algebra} is a commutative partial monoid $M$, equipped with a distinguished element $L \in M$, satisfying the following conditions.
    \begin{enumerate}
        \item (Orthocomplement) For every $x \in M$, there is a unique $x^\perp \in M$ such that $x \cdot x^\perp = L$.
        \item (Zero-one Law) If $x \cdot L$ is defined, then $x = e$.
    \end{enumerate}
\end{definition}
By Corollary \ref{cor:cyclicnerve}, the nerve of an effect algebra has a cosymmetric structure. This slightly enhances a result of Roumen \cite{roumen:effectalgebras}, who showed that the nerve of an effect algebra has a cyclic structure. We note that Roumen's construction was based on \emph{test spaces}, which were defined identically to $\hat{N}_n^L(M)$ in \eqref{eqn:test}.

Some of the examples we have seen previously are effect algebras, so their nerves have cosymmetric structures where the cyclic operation is given by \eqref{eqn:tau}.

\begin{example}
    The unit interval $M= [0,1]$ with the partially-defined addition operation and $L=1$ (see Example \ref{ex:interval}) is an effect algebra. Its nerve is given by
    \[ N_n M = \left\{(x_1,\dots,x_n) \in [0,1]^n \suchthat x_1 + \dots + x_n \leq 1\right\}.\]
    In this case, the isomorphism $\hat{N}_n^1 M \cong N_n M$ reproduces the classical isomorphism between the two standard models of the $n$-simplex.
\end{example}

\begin{example}
    The partial monoid $M = \{0, \dots, L\}$ with the partially-defined addition operation (see Example \ref{ex:L} is an effect algebra.
\end{example}

\begin{example}
    The power set $M=P(S)$ of a set $S$, with the partially-defined disjoint union operation (see Example \ref{ex:disjointunion}), is an effect algebra with $L=S$.
\end{example}

\begin{example}
 The set of subspaces of a Hilbert space $\mathcal{H}$, with the partially-defined orthogonal sum operation (see Example \ref{ex:orthogonalsum}), is an effect algebra with $L = \mathcal{H}$.
\end{example}

\subsection{Examples that are not effect algebras}
\label{sec:noteffect}

Here we will give some examples of $2$-Segal cosymmetric sets coming from applying the $L$-simplex construction to partial monoids that are not effect algebras. In particular, Example \ref{ex:union} produces pointed $2$-Segal cosymmetric sets that are not isomorphic to the nerve of any partial monoid.

\begin{example}
Let $G$ be an abelian group, and let $L \in G$ be any fixed element. Since the zero-one law is not satisfied (except when $G$ is trivial), $G$ is not an effect algebra. However, the orthocomplement property is satisfied, so Corollary \ref{cor:cyclicnerve} still holds, and the nerve of $G$ has a cosymmetric structure, where in particular the cyclic structure is given by
\[ \tau(g_1,\dots,g_n) = (g_2,\dots,g_n,L(g_1\cdots g_n)^{-1}).\]
\end{example}

\begin{example}\label{ex:union}
Let $S$ be a set, and consider the monoid $M = P(S)$ with the union operation. We stress that this is the ordinary union operation, different from the partially-defined disjoint union operation in Example \ref{ex:disjointunion}. 

Setting $L=S$, the $S$-simplex set $\hat{N}_\bullet^S M$ is given by
\[ \hat{N}_n^S M = \left\{(A_0,\dots,A_n)\in M^{n+1} \suchthat A_0 \cup \cdots \cup A_n = S.\right\}\]
The $L$-simplex set is a pointed $2$-Segal cosymmetric set, but since the orthocomplement property is not satisfied (unless $S = \emptyset$), the map in Proposition \ref{prop:simplextonerve} is not an isomorphism. In fact, $\hat{N}_\bullet^S$ is not isomorphic to the nerve of any partial monoid. To see this, we observe that the nerve of a partial monoid satisfies the property that the map $(d_2,d_0): X_2 \rightarrow X_1 \times X_1$ is injective, but in this case $d_2(A,S,S) = (A \cup S,S) = (S,S)$ and $d_0(A,S,S) = (A \cup S,S) = (S,S)$ for all $A \in P(S)$.
\end{example}

\section{Coherent 2D TQFTs}

As described in the Introduction, our motivating desire is to interpret a $2$-Segal cosymmetric set as a coherent $2$-dimensional TQFT taking values in $\Span$. In this section, we provide more detail on what this means, and we describe how one can obtain genuine TQFTs via the simplex set construction of Section \ref{sec:simplexdef}.

\subsection{Pseudofunctors from the cobordism category}

We first make a somewhat trivial observation about the well-known correspondence between commutative Frobenius algebras and $2$-dimensional oriented TQFTs. Given a commutative Frobenius algebra and a $2$-dimensional cobordism, if we want to produce the linear map that the associated TQFT assigns to the cobordism, then we first have to choose a \emph{pants decomposition}, i.e.\ a decomposition of the cobordism into pants, disks, and cylinders. We would then use data of the Frobenius algebra to assign a linear map to each piece of the decomposition and compose appropriately to obtain the linear map assigned to the cobordism. The axioms of a commutative Frobenius algebra are necessary and sufficient to ensure that the linear maps produced from two different pants decompositions of the same cobordism are equal.

Given a $2$-Segal cosymmetric set $X_\bullet$, we can similarly obtain a span from any cobordism equipped with a pants decomposition\footnote{We should also choose an order in which the compositions and products of spans are to be computed, but for simplicity we will assume that a standard order of operations is fixed.}. To do so, we use the multiplication and unit spans \eqref{eqn:multandunit} and the comultiplication and counit spans
    \begin{align*}
X_1 \xleftarrow{d_0} &X_2 \xrightarrow{(\tau d_2, d_1)} X_1 \times X_1, & X_1 \xleftarrow{} &X_0 \xrightarrow{s_1} \{\mathrm{pt}\}.
\end{align*}
The spans associated to two different pants decompositions of the same cobordism are not equal, but they are isomorphic via an isomorphism constructed out of the associator \eqref{eqn:associator}, unitors \eqref{eqn:unitors}, the maps $s_2^1$ and $\theta_1^2$, and the structure $2$-morphisms of $\Span$ (see \cite{stay}*{Section 5}). The coherence conditions should ensure that the isomorphism constructed in this way is unique\footnote{We add the caveat that making this rigorous would require a coherence theorem for commutative Frobenius pseudomonoids, which to our knowledge is not yet present in the literature.}.

If we globally make a choice of pants decomposition for every diffeomorphism class of cobordisms, then the data we have just described gives a pseudofunctor from the $2$-dimensional oriented cobordism category $\cob_2$ to $\Span$. Different global choices of decompositions lead to equivalent pseudofunctors, but we note that there is a standard choice, given by the ``normal form'' decomposition in \cite{kock}*{Section 1.4.16}).

If we pass to the homotopy category of $\Span$ by identifying spans that are isomorphic, then the above functor induces a TQFT taking values in the category of spans.

\subsection{The Hall algebra construction}

In \cite{BaezGroupoid}, the authors discuss functors which take spans (satisfying certain finiteness conditions) to linear maps of vector spaces. When these functors are applied to the data of a $2$-Segal set, we obtain an associative algebra. See \cite{Dyckerhoff-Kapranov:Higher}*{Chapter 8} or \cite{cooper-young} for details, and see \cite{MM:2segal}*{Section 6} for a concise summary of the various constructions and their relationships. Here, we will focus our attention on finite $2$-Segal sets, in which case the different functors coincide.

Let $X_\bullet$ be a $2$-Segal set that is finite (in the sense that $X_n$ is finite for each $n$), and let $\kk$ be an arbitrary field. The \emph{Hall algebra} $\mathcal{A}(X)$ is defined as follows. As a vector space, $\mathcal{A}(X) = \kk[X_1]$, and for $x,y \in X_1$ the multiplication $m(x,y)$ is given by
\begin{equation}\label{eqn:hallmult}
    m(x,y) = \sum_{\omega \in T^{-1}(x,y)} d_1 \omega,
\end{equation} 
where $T: X_2 \to X_1 \times X_1$ is given by $T(\omega) = (d_2 \omega, d_0 \omega)$. The multiplicative identity is 
\[ 1 = \sum_{u \in X_0} s_0 u.\]

As an immediate consequence of Corollary \ref{cor:commfrob}, if $X_\bullet$ has a cosymmetric structure, then $\mathcal{A}(X)$ is a Frobenius algebra, with counit given by
\[ \varepsilon(x) = \begin{cases}
    1 & x \in s_1(X_0),\\
    0 & x \not\in s_1(X_0),
\end{cases}\]
and the multiplication \eqref{eqn:hallmult} is commutative. Applying this fact to the simplex construction of Section \ref{sec:simplexdef}, we obtain a construction that produces a commutative Frobenius algebra from any finite commutative partial monoid equipped with a distinguished element.

\begin{example}
    Let $M$ be a finite commutative partial monoid that has orthocomplements with respect to $L \in M$. Then, by Corollary \ref{cor:cyclicnerve}, the nerve of $M$ has a cosymmetric structure. The associated Hall algebra is $\mathcal{A} = \kk[M]$, where the multiplication is given by
    \[ m(x,y) = \begin{cases}
        x \cdot y & \text{if $x \cdot y$ is defined},\\
        0 & \text{otherwise},
    \end{cases}\]
the identity element is $1 = e$, and the counit is given by
\[ \varepsilon(x) = \begin{cases}
    1 & x=L,\\
    0 & x\neq L.
\end{cases}\]
In special cases, we can explicitly describe the Frobenius algebras that arise:
\begin{itemize}
    \item When $M=G$ is an abelian group with distinguished element $L \in G$, this construction produces the group algebra $\kk[G]$ with counit given by $\varepsilon(L) = 1$ and $\varepsilon(g) = 0$ for $g \neq L$. In particular, when $G = \Z/m\Z$, the group algebra is isomorphic to $\kk[x]/\langle x^m - 1 \rangle$, where the counit extracts the $x^L$ coefficient.
    \item When $M=\{0,\dots,L\}$ with the partially-defined addition operation (see Example \ref{ex:L}), the associated Frobenius algebra $\mathcal{A}$ is isomorphic to $\kk[x]/\langle x^{L+1} \rangle$, where the counit extracts the $x^L$ coefficient.
    \item  When $S$ is a $k$-element set and $M = P(S)$ with the disjoint union operation (see Example \ref{ex:disjointunion}), $\mathcal{A}$ is isomorphic to $\kk[x_1,\dots,x_k]/\langle x_i^2\rangle$, where the counit extracts the $x_1\cdots x_k$ coefficient.
\end{itemize}
\end{example}

\begin{example}
    Let $S = \{a_1,\dots,a_k\}$ be a finite set, and let $M=P(S)$ with the union operation (see Example \ref{ex:union}). Then 
    \[\hat{N}_1^S M = \left\{(A_0,A_1) \in P(S)^2 \suchthat A_0 \cup A_1 = S\right\}\]
    and 
    \[\hat{N}_2^S M = \left\{(C_0,C_1,C_2) \in P(S)^3 \suchthat C_0 \cup C_1 \cup C_2 = S\right\},\]
    with face maps given by 
    \begin{align*}
    d_0(C_0,C_1,C_2) &= (C_0 \cup C_1, C_2),\\
    d_1(C_0,C_1,C_2) &= (C_0, C_1 \cup C_2),\\
    d_2(C_0,C_1,C_2) &= (C_0 \cup C_2, C_1).
    \end{align*}
    We note that any $(A_0,A_1) \in \hat{N}_1^S M$ is equivalent to a partition of $S$ into three disjoint sets $A_0 \smallsetminus (A_0 \cap A_1)$, $A_1 \smallsetminus (A_0 \cap A_1)$, and $A_0 \cap A_1$, and therefore $|\hat{N}_1^S M| = 3^k$.

Let $\mathcal{A} = \kk[\hat{N}_1^S M]$ be the Hall algebra associated to $\hat{N}_\bullet^S M$. To better understand the structure of $\mathcal{A}$, we first consider elements of $\hat{N}_1^S M$ that are of the form $x_A = (S,A)$, where $A \in P(S)$. If $A \cap B = \emptyset$, then $\omega =(S,A,B)$ is the unique element of $\hat{N}_2^S M$ such that $T(\omega) = (x_A, x_B)$. Writing $x_A x_B$ for the product given by \eqref{eqn:hallmult}, we have $x_A x_B = d_1 \omega = x_{A \cup B}$. In particular, if we write $x_i = x_{\{a_i\}}$, then we have
\[ x_A = \prod_{a_i \in A} x_i\]
for all $A \in P(S)$.

Next, we consider elements of $\hat{N}_1^S M$ that are of the form $y_A = (S\smallsetminus A,A)$. If $A \cap B = \emptyset$, then $\omega = (S \smallsetminus (A \cup B), A,B)$ is the unique element of $\hat{N}_2^S M$ such that $T(\omega) = (y_A, y_B)$, so $y_A y_B = d_1 \omega = y_{A \cup B}$. In particular, if we write $y_i = y_{\{a_i\}}$, then we have
\[ y_A = \prod_{a_i \in A} y_i\]
for all $A \in P(S)$.

We also have, if $A \cap B = \emptyset$, then $\omega = (S \smallsetminus B, A,B)$ is the unique element of $\hat{N}_2^S M$ such that $T(\omega) = (x_A,y_B)$, so $x_A y_B = d_1 \omega = (S\smallsetminus B, A \cup B)$. We note that every element of $\hat{N}_1^S M$ can be uniquely expressed in this form, so it follows that $\mathcal{A}$ is generated by $\{x_i,y_i\}$.

For each $i$, there are two elements, $(S, \{a_i\},\{a_i\})$ and $(S \smallsetminus \{a_i\}, \{a_i\},\{a_i\})$, in $T^{-1}(x_i,x_i)$. Thus, we have
\[ x_i^2 = (S, \{a_i\}) + (S \smallsetminus \{a_i\}, \{a_i\}) = x_i + y_i,\]
so $y_i = x_i^2 - x_i$. Additionally, $T^{-1}(x_i,y_i) = \emptyset$, so
$x_i y_i = x_i^3 - x_i^2 = 0$.

The above discussion, together with dimension counting, allows us to conclude that $\mathcal{A}$ is isomorphic to $\kk[x_1,\dots, x_k]/\langle x_i^3 - x_i^2 \rangle$. The counit is defined by the property that $\varepsilon(\emptyset, S)) = 1$ and $\varepsilon$ vanishes on other elements of $\hat{N}_1^S M$, but in terms of the generators $x_i$, it is given by $\varepsilon(x_1^2 \cdots x_k^2) = 1$, with $\varepsilon$ vanishing on terms of degree $< 2k$.
\end{example}

\bibliography{cosymmetric}
\end{document}